\documentclass{amsart}

\usepackage{epsfig}
\usepackage[all]{xy}
\usepackage{algorithmic}
\usepackage{algorithm}
\usepackage{enumerate}

\newtheorem{theorem}{Theorem}[section]

\newtheorem{lemma}[theorem]{Lemma}
\newtheorem{proposition}[theorem]{Proposition}
\theoremstyle{definition}
\newtheorem{definition}[theorem]{Definition}
\newtheorem{example}[theorem]{Example}
\theoremstyle{remark}
\newtheorem{remark}[theorem]{Remark}
\numberwithin{equation}{section}

\newcommand{\im}{\operatorname{Im}}
\newcommand{\SL}{\operatorname{SL}}

\newcommand{\Int}{\operatorname{Int}}
\newcommand{\Hom}{\operatorname{Hom}}
\newcommand{\Ext}{\operatorname{Ext}}
\newcommand{\Coker}{\operatorname{Coker}}
\newcommand{\cH}{\mathcal{H}}
\newcommand{\N}{\mathbb{N}}
\newcommand{\Z}{\mathbb{Z}}
\newcommand{\Q}{\mathbb{Q}}
\newcommand{\R}{\mathbb{R}}
\newcommand{\C}{\mathbb{C}}
\newcommand{\Symb}{\operatorname{Symb}}
\newcommand{\Div}{\operatorname{Div}}
\newcommand{\cD}{\mathcal{D}}

\title[Cohomological interpretation of quadratic modular symbols]{Cohomological interpretation of quadratic modular symbols}

\author[P.~Bayer]{Pilar Bayer$^*$}
\thanks{$^{*}$Partially supported by MTM2009-07024}
\address{Departament d'\`Algebra i Geometria, Universitat de Barcelona, Gran Via de les Corts Catalanes 585, 08007 Barcelona, Spain.}
\email{bayer@ub.edu}

\author[I. Blanco-Chac\'{o}n]{Iv\'{a}n Blanco-Chac\'{o}n$^{**}$}
\thanks{$^{**}$Partially supported by Academy of Finland grant 131745}
\address{Department of Mathematics and Systems Analysis, Aalto University,~Otakaari 1 M, Espoo FI-00076, Finland.}
\email{ivnblanco@gmail.com}

\author[A.~F.~Boix]{Alberto F.~Boix$^{***}$}
\thanks{$^{***}$Partially supported by MTM2010-20279-C02-01}
\address{Department of Economics and Business, Universitat Pompeu Fabra, Jaume I Building, Ramon Trias Fargas 25-27, 08005 Barcelona, Spain.}
\email{alberto.fernandezb@upf.edu}



\begin{document}

\begin{abstract}
Bayer and Blanco-Chac\'on have recently defined quadratic modular symbols for the Shimura curves $X(D,N)$ attached to Eichler orders of level $N$ of an indefinite quaternion $\Q$-algebra of discriminant $D$. In this paper, we give a cohomological interpretation of these quadratic modular symbols.
Explicit computations for the homology of some Shimura curves are also provided.
\end{abstract}

\maketitle

\section*{Introduction}
In \cite{blancobayer}, Bayer and Blanco-Chac\'{o}n have defined a $p$-adic distribution attached to an elliptic curve $E$ defined over $\mathbb{Q}$ and a quadratic imaginary point of the upper half-plane.
This $p$-adic distribution takes values in an infinite-dimensional $\mathbb{Q}_p$-Banach space and the corresponding Mellin-Mazur transform produces algebraic points on $E$ in the spirit of \cite{darmon2}.
This construction has been recently transferred to the setting of quaternionic Shimura curves $X(D,N)$ attached to  Eichler orders $\mathcal{O}(D,N)$ of level $N$ in the rational quaternion algebra of discriminant $D>1$.
A definition of $p$-adic $L$-functions attached to these curves, based on quadratic imaginary points,
has been proposed in \cite{bayerblanco2}.
The fact that the Hecke algebra does not act on the set of $\Gamma(D,N)$-orbits of quadratic imaginary points implies that the $p$-adic measures take values on $p$-adic Banach space of countable dimension, as in the elliptic curve setting.
\vskip 2mm
In the present paper, we carry out a cohomological study of these quadratic modular symbols.
The main tool for that purpose will be a systematic use of Grothendieck spectral sequences.
\vskip 2mm
Now, we provide an overview of the contents of this paper.
\vskip 2mm
In section \ref{descomponiendo caminos}, we design an algorithm to decompose the homology classes
that can be applied to the 73 arithmetic Fuchsian groups of signature $(1,e)$, classified by \,Takeuchi in \cite{tak},
as well as a second algorithm, based on the use of quadratic points,
suitable for the classical modular case.
\vskip 2mm
In section \ref{la famosa interpretacion cohomologica}, we recall with some detail the isomorphism of Ash and Stevens (\cite{ashstevens}) which allows to describe the cohomology with compact support of the modular curve in terms of classical modular symbols.
Finally, we develop a similar method to characterize the space of quadratic modular symbols under a cohomological point of view.

\section*{General notations and conventions}
Throughout the paper, $\mathcal{H}$ will denote the complex upper half-plane and $\mathcal{H}^*=\mathcal{H}\cup\mathbb{P}^1(\Q)$.
We refer to \cite{alsinabayer} for unexplained facts and terminology about quaternion algebras and quaternion orders.
We refer to \cite{katok} for unexplained facts and terminology about Fuchsian groups and hyperbolic geometry.

\section{Decomposition of closed paths}\label{descomponiendo caminos}

Let $H$ be an indefinite quaternion $\mathbb{Q}$-algebra of discriminant $D$. Choose a maximal order $\mathcal{O}$, which is unique up to conjugation.
Denote by $\mathcal{O}_1$ the group of elements of reduced norm 1 in $\mathcal{O}$ and
fix an embedding $\psi$ of $H$ in $\mathrm{M}(2,\mathbb{R})$ (see \cite{alsinabayer}).
Let $\Gamma_H=\psi(\mathcal{O}_1)$.
We denote by $\Gamma$ an arithmetic Fuchsian group of the first kind commensurable with $\Gamma_H$, for some $H$.
Let $X(\Gamma)$ be the corresponding projective Shimura curve.
For $D>1$, $\Gamma$ is cocompact, while for $D=1$ this is not the case and $X(\Gamma)$ is a modular curve.
The homology group $H_1(X(\Gamma)(\C),\R)$ contains the maximal lattice $H_1(X(\Gamma)(\C),\Z)$, whose elements are classes of closed paths in $X(\Gamma)(\C)$.

\begin{theorem}[cf.\,\cite{manin} for $D=1$, \cite{bayerblanco2} for $D>1$]
Denote by $E_{\Gamma}$ and $P_{\Gamma}$ the sets of elliptic and parabolic elements of $\Gamma$, respectively. Let $\Gamma'$ be the commutator subgroup of $\Gamma$. Let $\alpha\in\mathcal{H}$ if $\Gamma$ is cocompact or $\alpha\in\mathcal{H}^*$, otherwise. For any $g\in\Gamma$, define $\phi_{\alpha}(g):=\{\alpha,g(\alpha)\}_{\Gamma}\in H_1(X(\Gamma)(\mathbb{C}),\mathbb{Z})$.

\begin{enumerate}[(i)]

\item If $\Gamma$ is not cocompact then, for any $\alpha\in\mathcal{H}^*$, there is a short exact sequence of groups
\[
\xymatrix{0\ar[r]& \Gamma'E_{\Gamma}P_{\Gamma}\ar[r]& \Gamma\ar[r]^-{\phi_{\alpha}}& H_1(X(\Gamma)(\mathbb{C}),\mathbb{Z})\ar[r]& 0}.
\]

\item If $\Gamma$ is cocompact then, for any $\alpha\in\mathcal{H}$, there is a short exact sequence of groups
\[
\xymatrix{0\ar[r]& \Gamma'E_{\Gamma}\ar[r]& \Gamma\ar[r]^-{\phi_{\alpha}}& H_1(X(\Gamma)(\mathbb{C}),\mathbb{Z})\ar[r]& 0}.
\]

\end{enumerate}
In both cases, the map $\phi:=\phi_{\alpha}$ is independent of $\alpha$.
\label{key1}
\end{theorem}
From now on, we denote by $L$ the kernel of the map $\phi$.

\begin{remark}
Apart from theorem \ref{key1}, and from the fact that $H_1(X(\Gamma)(\mathbb{C}),\mathbb{Z})\simeq\mathbb{Z}^{2g}$,
being $g$ the genus of $X(\Gamma)$, saying something more precise about the structure of the homology of $X(\Gamma)(\mathbb{C})$ seems a difficult task. In some cases, it is possible to find a presentation of the group $\Gamma$ consisting of a set of matrices not belonging to $L$ together with some relations involving certain commutators (see \cite{tak}).
In \cite{chuman}, an algorithm is given to find a presentation for congruence groups,
and in \cite{alsinabayer} presentations and fundamental domains are provided for some arithmetic Fuchsian groups in the cocompact case.
Here, one of the problems we shall tackle is the decomposition of homology classes in terms of the generators of $\Gamma$ under the assumption that a system of generators is given.
\vskip 2mm
We emphasize that finding fundamental domains and generators is a different (and a non-equivalent) problem to decomposing homology classes which, on the other hand,
is equivalent to reducing points in the upper half-plane to a given fundamental domain.
\end{remark}

Fix $\tau\in\mathcal{H}^*$ ($\tau\in\mathcal{H}$, if $\Gamma$ is cocompact). By virtue of theorem \ref{key1} and the above remark, given $\omega\in H_1(X(\Gamma)(\mathbb{C}),\mathbb{Z})$, there exists $g\in\Gamma$ such that $\omega=\{\tau,g(\tau)\}_{\Gamma}$. In the modular case, in addition to theorem \ref{key1},
we can use Manin's continued fraction trick, which allows to decompose $\omega$ as a $\mathbb{Z}$-linear combination of a family of non-closed paths, namely, the Manin distinguished classes (see \cite{manin}).
\vskip 2mm
We look for an algorithm which decomposes a given $g\in \Gamma$ as a product of elements in a fixed set of generators of $\Gamma$, so that we can express $\omega$ as a $\mathbb{Z}$-linear combination of certain distinguished closed paths.
We develop such algorithm for the finite family of all the arithmetic Fuchsian groups of signature $(1;e)$.
\vskip 2mm
Additionally, we provide an algorithm which decomposes matrices of the modular group $\SL(2,\mathbb{Z})$ as products of powers of the generators $S=\begin{bmatrix}0 &-1\\1 &0\end{bmatrix}$ and $T=\begin{bmatrix}1 &1\\0 &1\end{bmatrix}$.
Such algorithm is different from the classical one, which uses the euclidean algorithm (see \cite{diamond}, for instance).
Our algorithm resembles Manin's continued fraction trick, although it is based on quadratic imaginary points, instead of the cusp $i\infty$.
It also yields as output a decomposition of $\omega$ as a $\mathbb{Z}$-linear combination of non-closed paths.

\subsection{Arithmetic Fuchsian groups of signature $(1;e)$}
We recall that all the fundamental domains for $\Gamma$ have the same number of non-accidental elliptic cycles.

\begin{definition}
The \emph{signature} of $\Gamma$ is the $(r+1)$-tuple $(g;e_1,\dots,e_r)$, where $g$ is the genus of $X(\Gamma)$, $r$ is the number of non-equivalent non-accidental elliptic cycles, and for any elliptic cycle $\mathcal{E}_k$, $e_k$ is the integer such that the sum of the angles at the vertices of $\mathcal{E}_k$ equals $2\pi/e_k$.
\end{definition}
For $g=\begin{bmatrix}a & b\\c & d\end{bmatrix}\in\Gamma$, if $g$ is not a homothety then denote by $I(g)$ the isometry circle of $g$, namely, the set $\{z\in\mathcal{H}\mid\quad\mid cz+d\mid =1\}$.
If $g$ is a homothety of factor $\lambda$,
then define $I(g)=\{z\in\mathcal{H}\mid\quad\mid\lambda z\mid =1\}$.
Denote by $\Ext(I(g))$ the exterior of $I(g)$ and by $\Int(I(g))$ the complement of $\Ext(I(g))$.
For any $\lambda\in\mathbb{R}$, $\lambda>0$, denote
\[
S(\lambda)=\{z\in\mathcal{H}\mid\,\,\lambda^{-1}\leq\mid z\mid\leq\lambda\}.
\]
If $h$ is a homothety of factor $\lambda$, then notice that the isometry circles of $h$ and $h^{-1}$ are parallel in the hyperbolic metric. Sometimes (cf.\,\cite{alsinabayer}), it is possible to find a system of generators $G$ of $\Gamma$ such that one of them is an hyperbolic homothety $h$ of factor $\lambda$ and a fundamental domain of the form
\[
\mathcal{F}=\bigcap_{g\in G\setminus\{h,h^{-1}\}}\Ext(I(g))\cap S(\lambda).
\]
We shall call $S(\lambda)$ the \emph{fundamental strip} of $\mathcal{F}$. We can construct such a fundamental domain, for instance, when $\Gamma$ is one of the 73 arithmetic Fuchsian groups of signature $(1;e)$, which were classified by Takeuchi in \cite{tak}.
These arithmetic Fuchsian groups admit a presentation of the form $\Gamma=\langle \alpha,\beta:(\alpha\beta\alpha^{-1}\beta^{-1})^e=\pm 1\rangle$, where $\alpha,\beta$ are hyperbolic elements.

\begin{proposition}[Sijsling, \cite{sij}]
Let $\Gamma$ be a cocompact arithmetic Fuchsian group of signature $(1;e)$ generated by $\alpha$ and $\beta$. Then, after a change of variables, we can suppose that $\alpha$ is a homothety of factor $\lambda$ and $\beta=\begin{bmatrix}a & b\\-b & a\end{bmatrix}$. Furthermore, the hyperbolic rectangle $\mathcal{F}=S(\lambda)\cap\Ext(I(\beta))\cap\Ext(I(\beta^{-1}))$ is a fundamental domain for $\Gamma$.
\label{generatorsTak}
\end{proposition}

Let $\Gamma$ be a cocompact arithmetic Fuchsian group of signature $(1;e)$ generated by $\alpha,\beta$. Given $\omega\in H_1(X(\Gamma)(\mathbb{C}),\mathbb{Z})$, by theorem \ref{key1}, we know that, for any $\tau\in\mathcal{H}$, there exists $g\in\Gamma$ such that $\omega=\{\tau,g(\tau)\}_{\Gamma}$. Since quadratic imaginary points contained in $\mathcal{H}$ are dense in $\mathcal{H}$, we may suppose, without loss of generality, that $\tau$ is a quadratic imaginary point contained in the interior of $\mathcal{F}$.
Then the paths $\omega_{\alpha}=\{\tau,\alpha(\tau)\}_{\Gamma}$ and $\omega_{\beta}=\{\tau,\beta(\tau)\}_{\Gamma}$ form a $\mathbb{Z}$-basis of $H_1(X(\Gamma)(\mathbb{C}),\mathbb{Z})$.
\vskip 2mm
Our aim is to decompose explicitly $\omega$ as $n_{\alpha}\omega_{\alpha}+n_{\beta}\omega_{\beta}$,
with $n_{\alpha},n_{\beta}\in\mathbb{Z}$.
For doing this, it is enough to express $g$ as a product of powers of $\alpha$ and $\beta$.
The idea is to multiply $g$ by the left by a suitable sequence of matrices $\{g_{k_j}\}$, with $g_{k_j}$ a power of $\alpha$ or $\beta$,
to obtain a product $g_{k_1}\cdots g_{k_n}g$, such that $g_{k_1}\cdots g_{k_n}g(\tau)$ belongs to the interior of $\mathcal{F}$. In this case, $g=(g_{k_1}g_{k_2}\cdots g_{k_n})^{-1}$. Observe that the decomposition of $g$ as a product of generators is not unique in general.

\begin{lemma}
Suppose that $\lambda\geq 1$. Then, for any $z\in\mathcal{H}$ there exists an integer $N$ such that $\lambda^{-1}\leq |\alpha^N(z)|\leq \lambda$.
\label{optimal}
\end{lemma}
\begin{proof}
If $|z|>\lambda$, then take
\[
N:=\min\{s\in\mathbb{N}\mid\; |\lambda^{-2s}z|\leq\lambda\}.
\]
Since $\alpha^{-N}(z)=\lambda^{-2N}z$, if $\lambda^{-1}\leq|\lambda^{-2N}z|$,
then we are done.
Otherwise, if $\lambda^{-1}>|\lambda^{-2N}z|$, then it would follow, multiplying by $\lambda^2$, that
\[
\lambda >\lambda^2 \cdot |\lambda^{-2N}z|=|\lambda^{-2(N-1)}z|,
\]
but this is a contradiction taking into account the choice of $N$.
The case when $|z|\leq \lambda^{-1}$ is analogue.
\end{proof}

Next we define the sets of transformations
$$
\begin{array}{ll}
\Gamma^{+} &=\{\beta,\beta\alpha,\beta\alpha^{-1},\beta\alpha\beta^{-1},\beta\alpha^{-1}\beta^{-1}\},
\\
&\\
\Gamma^{-} &=\{\beta^{-1},\beta^{-1}\alpha,\beta^{-1}\alpha^{-1},\beta^{-1}\alpha\beta^{-1},\beta^{-1}
\alpha^{-1}\beta^{-1}\}.
\end{array}
$$
They will play an important role in our algorithm.

\begin{example}
Figure \ref{dominio fundamental} shows a fundamental domain for the curve of signature $(1;2)$ labeled $e2d1D6ii$ in \cite{sij}.
In this particular case, $\alpha=\begin{bmatrix}\frac{\sqrt{6}+\sqrt{2}}{2} & 0\\0 & \frac{\sqrt{6}-\sqrt{2}}{2}\end{bmatrix}$ and $\beta=\begin{bmatrix}\sqrt{2} & 1\\1 & \sqrt{2}\end{bmatrix}$.
Figure \ref{translates} shows a detail of the right region of the stripe $S(\lambda)$,
in which we have depicted the images of $\mathcal{F}$ in figure \ref{dominio fundamental}
under the elements of $\Gamma^{+}$.
\end{example}

\begin{figure}
\centering
\scalebox{0.5}{\includegraphics{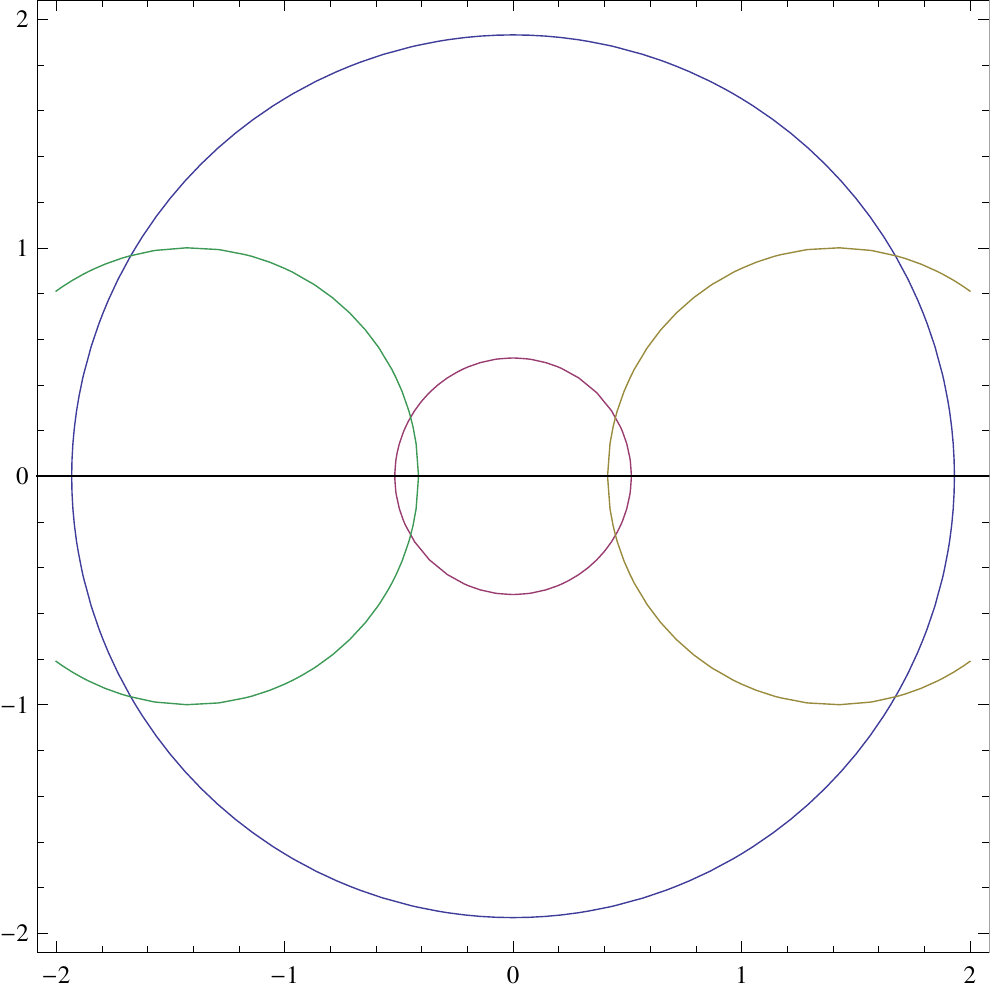}}
\caption{Fundamental domain in the upper half-plane for $\Gamma$.}
\label{dominio fundamental}
\end{figure}

\begin{figure}
\centering
\scalebox{0.5}{\includegraphics{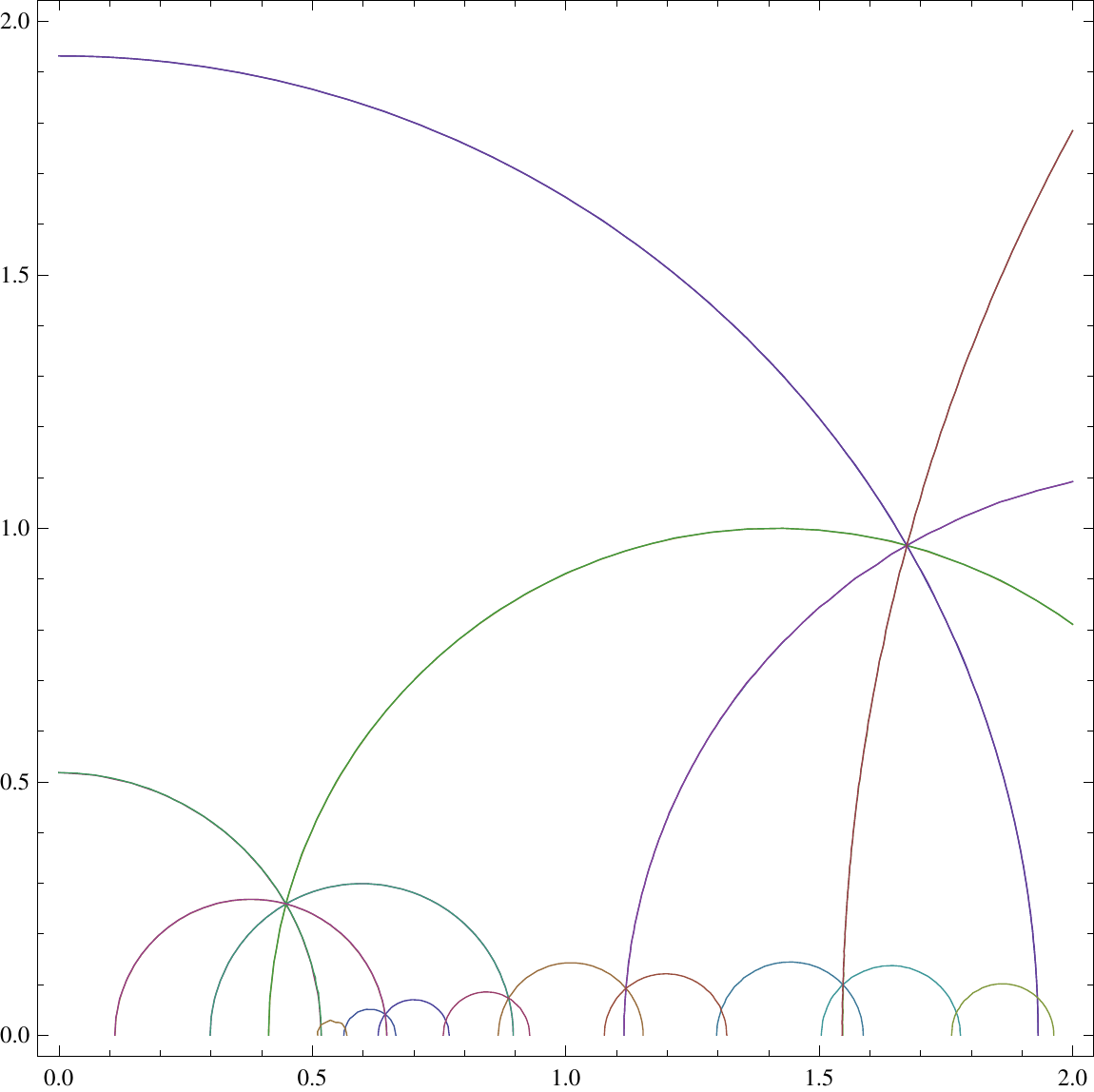}}
\caption{Right region in the case $e2d1D6ii$}
\label{translates}
\end{figure}
In general, denote by $S^{+}$ the right region of the stripe $S(\lambda)$ which excludes the interior of $\mathcal{F}$, and by $S^{-}$ the left portion of it, which is symmetric to $S^{+}$ with respect to the vertical axis.
Denote by $R^{+}$ and $R^{-}$ the regions covered by the images of $\mathcal{F}$ under iterated applications of the transformations in $\Gamma^{+}$ or $\Gamma^{-}$, respectively.
Set $$\mathcal{F}_n^+=\bigcup_{\gamma_{j_k}\in\Gamma^{+}}\gamma_{j_n}\gamma_{j_{n-1}}\cdots\gamma_{j_1}
(\mathcal{F}),$$
so that we have
\[
R^{+}=\bigcup_{n\geq 1}\mathcal{F}_n^+.
\]
The following results will also be used to prove the correctness of our algorithm.

\begin{lemma} For any $z\in S^{+}$, there exists $\gamma\in\Gamma^{+}$ such that $\im(\gamma(z))\leq \im(z)$.
\label{aux444}
\end{lemma}
\begin{proof}For $\gamma=\begin{bmatrix} a & b\\c & d\end{bmatrix}\in\Gamma^{+}$ and $z\in S^{+}$, we have
$$\im(\gamma(z))=\frac{\im(z)}{|cz+d|^2}.$$
Now, the condition $|cz+d|\geq 1$ is equivalent to say that $z$ belongs to the closure of the exterior of the isometry circle of $\gamma$.
\end{proof}

\begin{lemma}$S^{+}=R^{+}$.
\label{aux3}
\end{lemma}

\begin{proof}
First of all, let us check that, for any $n\geq 1$, and for any matrices $\gamma_{j_k}\in\Gamma^+$, is $\gamma_{j_n}\gamma_{j_{n-1}}\cdots\gamma_{j_1}(\mathcal{F})\subseteq S^{+}$. To see this, it is enough to notice that the five transforms of $\Gamma^{+}$ preserve $S^{+}$ (see figure \ref{translates}).
\vskip 2mm
To check the reciprocal inclusion,
define the sequence
\[
h_n:=\sup\left\{\im(z)\mid\; z\in\mathcal{F}_n^+\right\}.
\]
We claim that $h_{n+1}<h_n$ for any $n\geq 1$. In fact, since M\"{o}bius transforms are continuous, $h_n=\im(\gamma_{j_n}\cdots\gamma_{j_1}(z))$ for some $z\in\mathcal{F}$. In addition, it is easy to check (see figure \ref{translates}) that $h_n$ is reached at a point of the boundary which is intersection of two circles. Hence, $h_{n+1}$ is also reached at a boundary point which lies in the intersection of two circles.
\vskip 2mm
Let $V_n=\{z_{n,k}\}_{k}$ be the set of boundary points of $\mathcal{F}_n^+$ which are intersections of two circles. For any $k$, the point $z_{n,k}$ cannot be fixed simultaneously by the five matrices in $\Gamma^{+}$, hence, according to lemma \ref{aux444}, there exists $\gamma\in\Gamma^{+}$ such that $\im(\gamma(z_{n, k}))\leq\im(z_{n, k})$. Since M\"{o}bius transformations map geodesics into geodesics, the inequality has to be strict, and this for all the points in the finite set $V_n$. Since $h_{n+1}=\im(\gamma(z_{n,k_0}))$ for some boundary point $z_{n,k_0}\in V_n$, $h_{n+1}$ has to be strictly smaller than $h_n$.
\vskip 2mm
In this way, since $\{h_n\}_{n\geq 0}$ is a decreasing sequence and $h_n>0$ for any $n\geq 1$, this sequence is convergent. However, if the limit were nonzero then there would exist an accumulation point in $\mathcal{H}$ of a sequence $\{\gamma_n(z)\}\subseteq\mathcal{H}$ with $\gamma_n\in\Gamma$ and $z\in\mathcal{F}$. Since $\Gamma$ acts properly discontinuously on $\mathcal{H}$, this would be a contradiction.
\vskip 2mm
Consequently, given $z\in S^{+}$, let $n\geq 1$ be such that $h_n\leq\im(z)$. Hence, there exists $\gamma\in\Gamma$
equal to a product of $n$ matrices of $\Gamma^{+}$ such that $z\in\gamma(\mathcal{F})$.
\end{proof}

Now, we can turn the foregoing facts into an algorithm as follows.
\vskip 2mm
Let $\tau\in\mathcal{H}$ and $g\in\Gamma$ be such that $g(\tau)\not\in S(\lambda)$. Define $N(g)\in\mathbb{Z}$ such that $\lambda^{-1}\leq|\alpha^{N(g)}(\tau)|\leq\lambda$. Our procedure is described in algorithm \ref{hola1}. Let us move on check its correctness.

\begin{algorithm}\caption{Decomposition into distinguished closed paths}\label{hola1}
\begin{algorithmic}
\REQUIRE $g\in\Gamma$.
\ENSURE $\{n_{\alpha},n_{\beta}\}$ such that $\{\tau,g(\tau)\}=n_{\alpha}\{\tau,\alpha(\tau)\}+n_{\beta}\{\tau,\beta(\tau)\}$.
\medskip
\STATE $\gamma\gets g,n_{\alpha}\gets 0,n_{\beta}\gets 0$;
\STATE $flag=false$.
\WHILE {$flag==false$}
\IF{$\gamma(\tau)\not\in S$}
\STATE $n_{\alpha}\gets n_{\alpha}+N(g)$.
\STATE $g\gets\alpha^{N(g)}g$.
\STATE $\gamma\gets\gamma\alpha^{-N(g)}$.
\ELSE
\IF{$\gamma(\tau)\in\mathcal{F}$}
\STATE $flag\gets true$.
\ENDIF
\ENDIF
\IF{$\gamma(\tau)\in S^{+}$}
\STATE $n_{\beta}\gets n_{\beta}+1$.
\STATE $g\gets\beta^{-1}g$.
\STATE $\gamma\gets\gamma\beta$.
\ENDIF
\IF{$\gamma(\tau)\in S^{-}$}
\STATE $n_{\beta}\gets n_{\beta}-1$.
\STATE $g\gets\beta g$.
\STATE $\gamma\gets\gamma\beta^{-1}$.
\ENDIF
\ENDWHILE
\RETURN $\{n_{\alpha},n_{\beta}\}$ such that $\{\tau,g(\tau)\}=n_{\alpha}\{\tau,\alpha(\tau)\}+n_{\beta}\{\tau,\beta(\tau)\}$.
\end{algorithmic}
\end{algorithm}

\begin{proof}[Proof of correctness]
If $g(\tau)\not\in S(\lambda)$ then left multiplication by $\alpha^m$, for some $m$, brings $g(\tau)$ to, say, $S^{+}$. Then, by lemma \ref{aux3}, we can write $\alpha^{m}g(\tau)=\gamma_{k_n}\cdots\gamma_{k_1}(\tau)$, where $\gamma_{k_j}\in\Gamma^{+}$. If $\gamma_{k_n}\neq\beta\alpha\beta^{-1},\beta\alpha^{-1}\beta^{-1}$, then we can write
\[
\alpha^v\beta^{-1}\alpha^mg(\tau)=\gamma_{k_{n-1}}\cdots\gamma_{k_1}(\tau)
\]
for some $v\in\Z$. This point belongs to $S^{+}$. If $\gamma_{k_n}=\beta\alpha\beta^{-1}$ then we distinguish two cases. If $\gamma_{k_j}\neq\beta\alpha\beta^{-1},\beta\alpha^{-1}\beta^{-1}$ for some $j\in\{1,\dots ,n-1\}$ then
\[
\beta^{-1}\alpha^v\beta^{-1}\alpha^mg(\tau)=\gamma_{k_{j-1}}\cdots\gamma_{k_1}(\tau)\in S^{+}.
\]
Otherwise, $\alpha^v\beta^{-1}\alpha^mg(\tau)=\beta^{-1}(\tau)\in S^{-}$. The algorithm does $n$ operations if $g=\gamma_{k_n}\cdots\gamma_{k_1}$.
\end{proof}

\begin{example}
For $g=\begin{bmatrix}6+3\sqrt{3} &2\sqrt{2}\\2\sqrt{2} & 6-3\sqrt{3}\end{bmatrix}$, algorithm \ref{hola1} returns
$n_{\alpha}=n_{\beta}=2$ in four iterations. Indeed, $g=\alpha\beta^2\alpha$.
\end{example}

\subsection{An alternative algorithm for the modular case.}
We develop a method which, given a matrix $g\in\SL(2,\mathbb{Z})$, gives a factorization of $g$ in terms of $S$ and $T$. There exists an explicit algorithm which uses the euclidean algorithm (cf.\,\cite{serre}), but our approach is slightly different; indeed, we compare $g$ with the elements of a sequence of products of matrices acting on the imaginary unit in such a way that this sequence can be understood as the convergents of a certain continued fraction-like expansion.
\vskip 2mm
We start with $g=\begin{bmatrix}a & b\\c & d\end{bmatrix}\in\SL(2,\mathbb{Z})$. Since $(ST)^3=Id$ and $S(i)=i$, we can suppose that
\[
g=T^{n_k}ST^{n_{k-1}}S\cdots T^{n_2}ST^{n_1},
\]
with $n_1,\dots ,n_k$ to be determined. We want to express $\omega=\{i,g(i)\}_{\Gamma_0(N)}$ as a $\mathbb{Z}$-linear combination of the form
\[
\omega=\{i,ST^{n_1}(i)\}+\sum_{j=1}^{k-1}\{T^{n_j}S\cdots T^{n_1}(i),T^{n_{j+1}}ST^{n_j}S\cdots T^{n_1}(i)\}.
\]
Define the following finite sequence:
\[
g_1(i)=n_1+i,\quad g_{m+1}(i)=n_{m+1}-\dfrac{1}{g_m(i)}.
\]
Write $g_j=\begin{bmatrix} a_j & b_j\\c_j & d_j\end{bmatrix}\in\SL(2,\mathbb{Z})$. Notice that $g_{n_k}(i)=g(i)$. At the $j$-th step, $g_j(i)$ admits an explicit expression as an element of $\mathbb{Q}(i)$, which we call the $j$-th convergent. For instance, the first four convergents are
$$
\begin{array}{ccl}
g_1(i) & =& n_1+i\\
g_2(i) & =& \dfrac{n_2i+(n_1n_2-1)}{n_1+i}\\
g_3(i) & =& \dfrac{(n_2n_3-1)i+(n_1n_2n_3-n_1-n_3)}{n_2i+(n_1n_2-1)}\\
g_4(i) & =& \dfrac{(n_2n_3n_4-n_2-n_4)i +(n_1n_2n_3n_4-n_1n_4-n_3n_4-n_1n_2+1)}{(n_2n_3-1)i+(n_1n_2n_3-n_1-n_3)}.
\end{array}
$$
For any $j=2,\dots ,k$, notice that if $|a_j|>|c_j|$, the quotient $f_j=a_j/c_j$ can be written as $f_j=n_j-A_j/B_j$, where $A_j=B_{j-1}$, $B_j=n_{j-1}B_{j-1}-A_{j-1}$. We set $B_1=0$. Hence, for any $j\geq 2$, either $A_j$ is coprime to $B_j$, or $A_j=\pm B_j$.
\vskip 2mm
Now, we are ready for introducing our second algorithm.
\vskip 2mm
Given $a,b\in\mathbb{Z}$ with $b\neq 0$, denote by $\Div(a,b)$ the quotient of the euclidean division of $a$ by $b$. We propose the following procedure.

\begin{algorithm}\caption{Decomposition of matrices}\label{hola2}
\begin{algorithmic}
\REQUIRE $g\in\SL\left(2,\mathbb{Z}\right)$.
\ENSURE $v=\{n_k,\dots ,n_1\}$ such that $g(i)=T^{n_k}S\cdots T^{n_1}(i)$.
\medskip
\STATE $\gamma\gets g$.
\STATE $v\gets\emptyset$.
\WHILE {$\gamma\neq Id,S$}
\IF{$|\gamma[1,1]|<|\gamma[2,1]|$}
\STATE $\gamma\gets S\gamma$.
\STATE $h\gets hS$.
\ELSE
\STATE $\gamma\gets T^{-\Div(\gamma[1,1],\gamma[2,1])}\gamma$.
\STATE $v\gets v\cup\{\Div\Gamma[1,1],\gamma[2,1]\}$.
\ENDIF
\ENDWHILE
\RETURN $v=\{n_k,\dots ,n_1\}$ such that $g(i)=T^{n_k}S\cdots T^{n_1}(i)$.
\end{algorithmic}
\end{algorithm}

\begin{proof}[Proof of correctness]
If $c=0$, then $g$ is a translation. If $|a|>|c|$, then the integer part of the quotient $f_k$ gives the exponent $n_k$, otherwise, the upper left entry of $Sg$ is bigger in absolute value than the lower-left entry.
\end{proof}
\begin{example}
For $g=\begin{bmatrix}3 &2\\7 & 5\end{bmatrix}$, algorithm \ref{hola2} returns the ordered vector $v=\{-2,3,1\}$ in four iterations. Indeed, $g=-ST^{-2}ST^3ST$.
\end{example}

\section{Quadratic modular symbols}\label{la famosa interpretacion cohomologica}
The aim of this section is to present the quadratic modular symbols as the cohomology of a suitable pair of topological spaces.
\vskip 2mm
As a convention in what follows, given any topological space $X$ and any $i\in\N$, we write $H_i (X)$ instead of $H_i (X,\Z)$. Similarly, for any pair $A\subseteq X$ of topological spaces one writes $H_i ((X,A))$ instead of $H_i((X,A),\Z)$.

\subsection{Ash-Stevens cohomological interpretation of the classical modular symbols}
We start by fixing some additional notation.
\vskip 2mm
Set $\Gamma :=\Gamma_0 (N)$, let $\cD:=\Div (\mathbb{P}^1 (\Q))$ be the group of divisors supported on the rational cusps $\mathbb{P}^1 (\Q)=\Q\cup\{i\infty\}$, let $\cD_0\subset\cD$ be the subgroup of divisors of degree zero, let $R$ be a commutative ring such that the order of every torsion element of $\Gamma$ is invertible in $R$ and let $E$ be an $R[\Gamma]$-module. We can identify $\Gamma\backslash \cH^*$ with $X(\Gamma)(\mathbb{C})$. Denote $Y(\Gamma)(\mathbb{C})$ for $\Gamma\backslash\mathcal{H}$. Let $\widetilde{E}$ be the local coefficient system on $X(\Gamma)(\C)$ associated to $E$.
\begin{theorem}[Ash-Stevens, \cite{ashstevens}]
For any $i\in\N$ we have the following commutative diagram with exact rows:
\[
\scalebox{0.8}{
\xymatrix{H^i \left(\left(X(\Gamma)(\mathbb{C}),X(\Gamma)(\mathbb{C})\setminus Y(\Gamma)(\mathbb{C}) \right),\widetilde{E}\right)\ar[d]\ar[r]& H^i \left(X(\Gamma)(\mathbb{C}),\widetilde{E}\right)\ar[d]\ar[r]& H^i \left(X(\Gamma)(\mathbb{C})\setminus Y(\Gamma)(\mathbb{C}),\widetilde{E}\right)\ar[d]\\
H^{i-1}(\Gamma ,\Hom_{\Z}(\cD_0 ,E))\ar[r]& H^i (\Gamma ,E)\ar[r]& H^i(\Gamma ,\Hom_{\Z}(\cD ,E))}
}
\]
where the vertical arrows are isomorphisms.
\label{ashstevenscohomology}
\end{theorem}

\begin{proof}
First of all, set $A:=\cH^*\setminus\cH$. We have the following commutative diagram with exact rows:
\[
\scalebox{0.78}{
\xymatrix{Q\ar[d]\ar@{^{(}->}[r]& H^0 (\cH^* ,E)\ar[d]\ar[r]& H^0 (A,E)\ar[d]\ar[r]& H^1 \left((\cH^* ,A),E\right)\ar[d]\ar[r]& H^1 (\cH^* ,E)\ar[d]\\
\Hom_{\Z}(K,E)\ar@{^{(}->}[r]& \Hom_{\Z}(H_0 \left(\cH^*\right) ,E)\ar[r]& \Hom_{\Z}(H_0 \left(A\right) ,E)\ar[r] &  G_1\ar[r]&  G_2}
}
\]
where
\begin{align*}
& Q:=H^0 \left((\cH^* ,A),E\right),\quad K:=H_0 \left((\cH^* ,A)\right),\\
& G_1:=\Hom_{\Z}(H_1 \left((\cH^* ,A),E\right))\oplus\Ext_{\Z}^1\left(H_0 \left((\cH^* ,A)\right),E\right),\\
& G_2:=\Hom_{\Z}(H_1 \left(\cH^*\right),E)\oplus\Ext_{\Z}^1\left(H_0 \left(\cH^*\right),E\right).
\end{align*}
Moreover, the vertical arrows in the above diagram are the isomorphisms given by the universal coefficients theorem for cohomology. In addition, we underline that the following facts hold.
\begin{enumerate}[(i)]

\item $Q=0$. Indeed, observe that
\begin{equation*}
H_0 \left((\cH^* ,A)\right)=\Coker(\xymatrix{H_0 (A)\ar[r]^-{i_*}& H_0 (\cH^*)}).
\end{equation*}
Since the boundary components of $\cH^*$ are in one-one correspondence with $\mathbb{P}^1 (\Q)$ it follows that $H_0 (A)\cong\cD$. Moreover, $H_0 (\cH^*)\cong\Z$ since $\cH^*$ is pathwise connected. With these identifications in mind, $i_*$ (which is the map induced in homology by the inclusion $\xymatrix{A\ar@{^{(}->}[r]^-{i}& \cH^*}$) can be identified with the degree map $\xymatrix{\cD\ar[r]& \Z}$, which is clearly surjective. Hence
\[
H_0((\mathcal{H}^*,A),E)=H_0 \left((\cH^* ,A)\right)\otimes_{\Z} E\cong 0.
\]

\item $\Hom_{\Z}(H_0 (\cH^*),E)\cong E$. Indeed, this fact is obvious because $\cH^*$ is pathwise connected.

\item $\Ext_{\Z}^1\left(H_0 \left(\cH^*\right),E\right)=0$. Indeed, as $\cH^*$ is pathwise connected and $\mathbb{Z}$ is projective it follows that $\Ext_{\Z}^1\left(H_0 \left(\cH^*\right),E\right)=\Ext_{\Z}^1\left(\Z,E\right)=0$.

\item $H_1 \left(\cH^*\right)=0$. Indeed, as $\cH^*$ is pathwise connected $H_1 \left(\cH^*\right)$ is the abelianization of the first fundamental group $\pi_1 (\cH^*)$. But $\pi_1 (\cH^*)\cong 1$ because $\cH^*$ is simply connected.

\end{enumerate}
In this way, combining all these facts, the above diagram becomes into the following commutative diagram with exact rows and vertical isomorphisms.
\[
\xymatrix{0\ar[r]& H^0 (\cH^* ,E)\ar[d]\ar[r]& H^0 (A,E)\ar[d]\ar[r]& H^1 \left((\cH^* ,A),E\right)\ar[d]\ar[r]& 0\\
0\ar[r]& E\ar[r]& \Hom_{\mathbb{Z}}(\cD,E)\ar[r]& \Hom_{\mathbb{Z}}(\cD_0,E)\ar[r]& 0.}
\]
Applying to the above diagram the left exact functor $(-)^{\Gamma}$ we obtain the following commutative diagram with exact rows and vertical isomorphisms.
\[
\xymatrix{H^i (\Gamma ,H^0(\cH^* ,E))\ar[d]\ar[r]& H^i (\Gamma ,H^0(A ,E))\ar[d]\ar[r]& H^i \left(\Gamma ,H^1 \left((\cH^* ,A),E\right)\right)\ar[d]\\
H^i (\Gamma ,E)\ar[r]& H^i (\Gamma ,\Hom_{\mathbb{Z}}(\cD,E))\ar[r]& H^i (\Gamma ,\Hom_{\mathbb{Z}}(\cD_0,E)).}
\]
Thus, we only need to check, for each $i\in\N$, that
\begin{align*}
& H^i (\Gamma ,H^0(\cH^* ,E))\cong H^i (X(\Gamma)(\mathbb{C}) ,\widetilde{E}),\\
& H^i (\Gamma ,H^0(A ,E))\cong H^i (X(\Gamma)(\mathbb{C})\setminus Y(\Gamma)(\mathbb{C}) ,\widetilde{E}),\mbox{ and }\\
& H^i (\Gamma ,H^1 \left((\cH^* ,A),E\right))\cong H^{i+1} \left(\left(X\left(\Gamma)(\mathbb{C}),X(\Gamma)(\mathbb{C})\setminus Y(\Gamma)(\mathbb{C}\right)\right),\widetilde{E}\right).
\end{align*}
Now, consider the following Grothendieck spectral sequences.
\begin{align*}
& \xymatrix{H^i (\Gamma ,H^j(\cH^* ,E))\ar@{=>}[r]_-{i}& H^{i+j} (X(\Gamma)(\mathbb{C}),\widetilde{E})},\\
& \xymatrix{H^i (\Gamma ,H^j(A ,E))\ar@{=>}[r]_-{i}& H^{i+j} (X(\Gamma)(\mathbb{C})\setminus Y(\Gamma)(\mathbb{C}),\widetilde{E})},\\
& \xymatrix{H^i (\Gamma ,H^j(\left(\cH^* ,A\right),E))\ar@{=>}[r]_-{i}& H^{i+j} ((X(\Gamma)(\mathbb{C}),X(\Gamma)(\mathbb{C})\setminus Y(\Gamma)(\mathbb{C})),\widetilde{E})}.
\end{align*}
We start by analyzing the first one. The foregoing calculations show, in particular, that $H^j(\cH^* ,E)=0$ for all $j\neq 0,2$. This fact implies that the source or the target of any differential of the $E_2$-page of this spectral sequence is zero and therefore it collapses providing an isomorphism
\[
H^i (\Gamma ,H^j(\cH^* ,E))\cong H^{i+j} (X(\Gamma)(\mathbb{C}) ,\widetilde{E}),
\]
for any $(i,j)\in\N^2$.
\vskip 2mm
Let us move on to the second spectral sequence. Again by the foregoing, $H^j (A,E)=0$ for all $j\neq 0$. Thus, this spectral sequence has just one non-zero row, hence it collapses yielding an isomorphism
\[
H^i (\Gamma ,H^j(A ,E))\cong H^{i+j} (X(\Gamma)(\mathbb{C})\setminus Y(\Gamma)(\mathbb{C}) ,\widetilde{E}),
\]
for any $(i,j)\in\N^2$.
\vskip 2mm
Bearing in mind that $H^j(\left(\cH^* ,A\right),E)=0$ for all $j\neq 1$, a similar argument implies that, for any $(i,j)\in\N^2$, there is an isomorphism
\[
H^i (\Gamma ,H^1 \left((\cH^* ,A),E\right))\cong H^{i+1} \left(\left(X(\Gamma)(\mathbb{C}),X(\Gamma)(\mathbb{C})\setminus Y(\Gamma)(\mathbb{C})\right),\widetilde{E}\right),
\]
just what we wanted to show.
\end{proof}
Notice that, in particular, $H^i (\Gamma ,H^1 \left((\cH^* ,A),E\right))\cong H_c^{i+1}(Y(\Gamma)(\mathbb{C}),\widetilde{E})$.

\subsection{Quadratic modular symbols}
Let $\Gamma$ be an arithmetic Fuchsian group of the first kind and fix $\tau\in\mathcal{H}^*$ ($\tau\in\mathcal{H}$ if $\Gamma$ is cocompact). Let $p$ be a prime number and let $\omega_p\in\mathcal{O}_H$ be a quaternion of reduced norm $p$. Set $\gamma_p :=\psi (\omega_p)$. We must recall (cf. \cite{shimura}) that there exists $d\in\mathbb{N}$ such that
\[
\left[\Gamma :\Gamma\cap\gamma_p \Gamma\gamma_p^{-1}\right]=d
\]
and therefore there is a coset decomposition
\[
\Gamma\gamma_p \Gamma =\bigcup_{a=1}^d \gamma_a \Gamma .
\]
Set, for each $n\in\mathbb{N}$, $I_n:=\{1,\dots,d\}^n$. Moreover, for each $\mathbf{u}=(u_1,\dots,u_n)\in I_n$ set
\[
\gamma_{\mathbf{u}}:=\prod_{t=1}^n \gamma_{u_t}.
\]
Finally, we define
\[
\Delta_{\tau ,p}:=\bigcup_{n\geq 0}\bigcup_{\mathbf{u}\in I_n}\gamma_{\mathbf{u}} \Gamma \tau .
\]
The reason why we define $\Delta_{\tau,p}$ is, on one hand, that we are interested in considering integrals along geodesics connecting points of $\Gamma\tau$ (classical modular symbols are integrals along geodesics connecting cusps), and, on the other hand, we want that the action of the coset representatives given by the matrices $\gamma_j$ respects our module of values. Notice that when $\tau=i\infty$, we have again $\Delta_{i\infty,p}=\mathbb{P}^1\left(\mathbb{Q}\right)$.
\vskip 2mm
From now on, we suppose that $\Gamma$ is cocompact and that $\tau$ is a quadratic imaginary point.
\vskip 2mm
Consider the topological pair $(\mathcal{H},\Delta_{\tau,p})$. We are interested in finding a result analogous to theorem \ref{ashstevenscohomology} in the quadratic setting.

Consider the inclusion $\xymatrix{\Delta_{\tau,p}\ar@{^{(}->}[r]^-{i}&\mathcal{H}}$ and, for any abelian group $G$, recall that
\[
H_0 \left((\mathcal{H},\Delta_{\tau,p}),G\right)=\Coker(\xymatrix{H_0 (\Delta_{\tau,p},G)\ar[r]^-{i_*}& H_0 (\cH,G)}),
\]
where $i_*$ is the map induced in $0$-th homology by the inclusion $i$.
\vskip 2mm
Consider again the beginning of the long exact sequence in homology of the pair $(\cH ,\Delta_{\tau ,p})$,
\[
\xymatrix{H_1 (\cH)\ar[r]& H_1 ((\cH ,\Delta_{\tau ,p}))\ar[r]& H_0 (\Delta_{\tau ,p})\ar[r]^-{i_*}& H_0 (\cH)\ar@{->>}[r]& H_0 ((\cH ,\Delta_{\tau ,p}))}.
\]
As $\cH$ is simply connected, $H_1 (\cH)=0$. On the other hand, as $\cH$ is pathwise connected $H_0 (\cH)=\Z$ and therefore the map $i_*$ is the map induced in homology by the degree map
\begin{align*}
& \xymatrix{C_0 (A)\ar[r]& \Z}\\
& \sum_i b_i \sigma_i\longmapsto\sum_i b_i.
\end{align*}
Here, $C_0 (A)$ denotes the free abelian group of $0$-dimensional singular simplices of $A$. In this way, it follows from this fact that $i_*$ is surjective. Hence we have the short exact sequence of abelian groups
\[
\xymatrix{0\ar[r]& H_1 ((\cH ,\Delta_{\tau ,p}))\ar[r]^-{\partial_{\tau,p}}& H_0 (\Delta_{\tau ,p})\ar[r]^-{i_*}& H_0 (\cH)\ar[r]& 0.}
\]
Set $D_{\tau,p}:=H_0\left(\Delta_{\tau,p}\right)$ and $D_{\tau,p}^0:=H_1\left(\left(\mathcal{H},\Delta_{\tau,p}\right)\right)$. Underline that the above exact sequence allows us to view $D_{\tau,p}^0$ as a subgroup of $D_{\tau,p}$. If we interpret $D_{\tau,p}$ as a group of divisors then, the foregoing implies that $D_{\tau,p}^0$ is the subgroup of divisors of degree zero.
\vskip 2mm
Since we are interested in studying modular symbols of weight 2, we shall consider $G=\mathbb{C}$. Denote
\begin{align*}
& Q_1:=\Hom_{\Z}(H_1((X,A)),\mathbb{C})\oplus\Ext_{\Z}^1 (H_0((X,A)),\mathbb{C}),\\
& Q_2:=\Hom_{\Z}(H_1(X),\mathbb{C})\oplus\Ext_{\Z}^1 (H_0(X),\mathbb{C}).
\end{align*}
Then, we have the following commutative diagram with exact rows, where the vertical arrows are the isomorphisms given by the universal coefficients theorem for cohomology:
\[
\scalebox{0.75}{
\xymatrix{H^0 ((X,A),\mathbb{C})\ar[d]\ar@{^{(}->}[r]& H^0 (X,\mathbb{C})\ar[d]\ar[r]& H^0 (A,\mathbb{C})\ar[d]\ar[r]& H^1 ((X,A),\mathbb{C})\ar[d]\ar[r]& H^1 (X,\mathbb{C})\ar[d]\\
\Hom_{\Z}(H_0((X,A)),\mathbb{C})\ar@{^{(}->}[r]& \Hom_{\Z}(H_0(X),\mathbb{C})\ar[r]& \Hom_{\Z}(H_0(A),\mathbb{C})\ar[r]& Q_1 \ar[r]& Q_2.}
}
\]
If we particularize the above diagram to our pair then we obtain the following result.
\begin{proposition}
There is a commutative diagram with exact rows, where the vertical arrows are the isomorphisms given by the universal coefficients theorem for cohomology:
\[
\xymatrix{0\ar[r]& H^0\left(\mathcal{H},\mathbb{C}\right)\ar[d]\ar[r] & H^0\left(\Delta_{\tau,p},\mathbb{C}\right)\ar[d]\ar[r] & H^1\left((\mathcal{H},\Delta_{\tau,p}),\mathbb{C}\right)\ar[d]\ar[r] & 0\\
0\ar[r]& \Hom_{\mathbb{Z}}\left(H_0\left(\mathcal{H}\right),\mathbb{C}\right)\ar[r]& \Hom_{\mathbb{Z}}\left(D_{\tau,p},\mathbb{C}\right)\ar[r]& \Hom_{\mathbb{Z}}\left(D_{\tau,p}^0,\mathbb{C}\right)\ar[r]& 0.}
\]
\end{proposition}

\begin{proof}
All we have to see is that the Ext terms are zero. First of all, as we know that $H_0\left(\left(\mathcal{H},\Delta_{\tau,p}\right)\right)=0$ it follows that $Q_1=0$. Secondly, since $\mathcal{H}$ is pathwise connected, we have that $Q_2=\Ext_{\mathbb{Z}}^1\left(H_0(\mathcal{H}),\mathbb{C}\right)=0$. Notice that the left terms in each row are isomorphic to $\mathbb{C}$.
\end{proof}
By passing to the long exact sequence (acting $\Gamma$ on each term), we obtain the following commutative diagram with exact rows and vertical isomorphisms:
\begin{equation}\label{exactalarga}
\xymatrix{H^{i-1}\left(\Gamma,\Hom_{\Z}\left(D_{\tau,p}^0,\mathbb{C}\right)\right)\ar[d]\ar[r] & H^i\left(\Gamma,\mathbb{C}\right) \ar[d]\ar[r]& H^i\left(\Gamma,\Hom_{\Z}\left(D_{\tau,p},\mathbb{C}\right)\right)\ar[d]\\
H^{i-1}\left(\Gamma,H^1\left(\left(\mathcal{H},\Delta_{\tau,p}\right),\mathbb{C}\right)\right)\ar[r] & H^i\left(\Gamma,H^0(\mathcal{H},\mathbb{C})\right)\ar[r]& H^i\left(\Gamma,H^0\left(\Delta_{\tau,p},\mathbb{C}\right)\right).}
\end{equation}

\begin{definition}Let $\tau\in\mathcal{H}$ be a quadratic imaginary point. The \emph{$\C$-valued quadratic modular symbols attached to $\tau$} is the $\C$-vector space
\[
H^0\left(\Gamma,\Hom_{\Z}\left(D_{\tau,p}^0,\mathbb{C}\right)\right).
\]
We shall denote this space in what follows by $\Symb\left(\Delta_{\tau,p},\C\right)^{\Gamma}$.
\end{definition}
Now, we consider the following Grothendieck spectral sequences
\begin{align*}
& \xymatrix{H^i (\Gamma ,H^j(\cH ,\C))\ar@{=>}[r]_-{i}& H^{i+j} (X(\Gamma)(\mathbb{C}),\C)},\\
& \xymatrix{H^i (\Gamma ,H^j(\Delta_{\tau ,p} ,\C))\ar@{=>}[r]_-{i}& H^{i+j} (\Gamma\backslash\Delta_{\tau ,p} ,\C)},\\
& \xymatrix{H^i (\Gamma ,H^j(\left(\cH ,\Delta_{\tau ,p}\right),\C))\ar@{=>}[r]_-{i}& H^{i+j} ((X(\Gamma)(\mathbb{C}),\Gamma\backslash\Delta_{\tau ,p}),\C).}
\end{align*}
We start by analyzing the first one. The foregoing calculations imply, in particular, that $H^j(\cH ,\C)=0$ for all $j\neq 0$. This fact implies that this spectral sequence has just one non-zero row and therefore it collapses providing an isomorphism
\[
H^i (\Gamma ,H^j(\cH ,\C))\cong H^{i+j} (X(\Gamma)(\C) ,\C),
\]
for any $(i,j)\in\N^2$.
\vskip 2mm
Let us move on to the second spectral sequence. Again by the foregoing, $H^j (\Delta_{\tau ,p},\C)=0$ for all $j\neq 0$. Thus, this spectral sequence has just one non-zero row, hence it collapses yielding an isomorphism
\[
H^i (\Gamma ,H^j(\Delta_{\tau ,p} ,\C))\cong H^{i+j} (\Gamma\backslash\Delta_{\tau ,p} ,\C),
\]
for any $(i,j)\in\N^2$.
\vskip 2mm
Bearing in mind that $H^j(\left(\cH ,\Delta_{\tau ,p}\right),\C)=0$ for all $j\neq 1$, a similar argument implies that, for any $(i,j)\in\N^2$, there is an isomorphism
\[
H^i (\Gamma ,H^1 \left((\cH ,\Delta_{\tau ,p}),\C\right))\cong H^{i+1} \left(\left(X(\Gamma)(\C),\Gamma\backslash\Delta_{\tau ,p}\right),\C\right).
\]
All the foregoing facts, together with diagram \eqref{exactalarga}, allow us to establish the main result of this section.

\begin{theorem}
The space of $\C$-valued quadratic modular symbols $\Symb\left(\Delta_{\tau,p},\C\right)^{\Gamma}$ attached to $\tau$ is canonically isomorphic to $H^1 \left(\left(X(\Gamma)(\C) ,\Gamma\backslash\Delta_{\tau ,p}\right),\C\right)$ as $\C$-vector space.
\end{theorem}

\bibliographystyle{amsplain}
\bibliography{versionspringer}

\providecommand{\bysame}{\leavevmode\hbox to3em{\hrulefill}\thinspace}
\providecommand{\MR}{\relax\ifhmode\unskip\space\fi MR }
\providecommand{\MRhref}[2]{%
  \href{http://www.ams.org/mathscinet-getitem?mr=#1}{#2}
}
\providecommand{\href}[2]{#2}
\begin{thebibliography}{10}

\bibitem{alsinabayer}
M.~Alsina and P.~Bayer, \emph{Quaternion orders, quadratic forms, and {S}himura
  curves}, CRM Monograph Series, vol.~22, American Mathematical Society,
  Providence, RI, 2004. \MR{2038122 (2005k:11226)}

\bibitem{ashstevens}
A.~Ash and G.~Stevens, \emph{Modular forms in characteristic {$\ell$} and
  special values of their {$L$}-functions}, Duke Math. J. \textbf{53} (1986),
  no.~3, 849--868. \MR{860675 (88h:11036)}

\bibitem{bayerblanco2}
P.~Bayer and I.~Blanco Chac\'on, \emph{Quadratic modular symbols on {S}himura
  curves}, Submitted. Available at {\tt
  http://www.arxiv.org/pdf/1112.5645.pdf}.

\bibitem{blancobayer}
\bysame, \emph{Quadratic modular symbols}, RACSAM \textbf{106} (2012), no.~2,
  429--441.

\bibitem{darmon2}
M.~Bertolini and H.~Darmon, \emph{Heegner points, {$p$}-adic {$L$}-functions,
  and the {C}erednik-{D}rinfeld uniformization}, Invent. Math. \textbf{131}
  (1998), no.~3, 453--491. \MR{1614543 (99f:11080)}

\bibitem{chuman}
Y.~Chuman, \emph{Generators and relations of {$\Gamma \sb{0}(N)$}}, J. Math.
  Kyoto Univ. \textbf{13} (1973), 381--390. \MR{0348001 (50 $\#$499)}

\bibitem{diamond}
F.~Diamond and J.~Shurman, \emph{A first course in modular forms}, Graduate
  Texts in Mathematics, vol. 228, Springer-Verlag, New York, 2005. \MR{2112196}

\bibitem{katok}
S.~Katok, \emph{Fuchsian groups}, Chicago Lectures in Mathematics, University
  of Chicago Press, Chicago, IL, 1992. \MR{1177168 (93d:20088)}

\bibitem{manin}
Ju.~I. Manin, \emph{Parabolic points and zeta functions of modular curves},
  Izv. Akad. Nauk SSSR Ser. Mat. \textbf{36} (1972), 19--66. \MR{0314846 (47
  $\#$3396)}

\bibitem{serre}
J.-P. Serre, \emph{A course in arithmetic}, Graduate Texts in Mathematics,
  vol.~7, Springer-Verlag, New York, 1973. \MR{0344216 (49 $\#$8956)}

\bibitem{shimura}
G.~Shimura, \emph{Introduction to the arithmetic theory of automorphic forms},
  Publications of the Mathematical Society of Japan, vol.~11, Iwanami Shoten,
  Publishers, Tokyo and Princeton University Press, Princeton, N.J, 1971,
  Kan\^{o} Memorial Lectures, No. 1. \MR{0314766 (47 $\#$3318)}

\bibitem{sij}
J.~R. Sijsling, \emph{Equations for arithmetic pointed tori}, Ph.D. thesis,
  Universiteit Utrecht, 2010, available at {\tt
  http://sites.google.com/site/sijsling/research}.

\bibitem{tak}
K.~Takeuchi, \emph{Arithmetic {F}uchsian groups with signature {$(1;e)$}}, J.
  Math. Soc. Japan \textbf{35} (1983), no.~3, 381--407. \MR{702765 (84h:10031)}

\end{thebibliography}

\end{document}